\documentclass[12pt]{amsart}
\usepackage[utf8]{inputenc}
\usepackage{t1enc}
\usepackage[margin=3.5cm]{geometry}
\usepackage{hyperref}
\usepackage{amsfonts,amsthm,amssymb,amsmath,amscd}
\usepackage{mathtools}
\numberwithin{equation}{section}

\theoremstyle{plain}
\newtheorem{thrm}{Theorem}
\newtheorem{lemma}[thrm]{Lemma}
\newtheorem{cor}[thrm]{Corollary}
\newtheorem{prop}[thrm]{Proposition}
\theoremstyle{definition}

\newcommand{\R}{\mathbb{R}}
\renewcommand{\Re}{\mathbb{R}}
\newcommand{\Red}{\R^d}

\newcommand{\aff}[1]{\mathrm{aff}\!\left(#1\right)}
\newcommand{\relint}[1]{\mathrm{relint}\!\left(#1\right)}
\newcommand{\inter}[1]{\mathrm{int}\!\left(#1\right)}
\newcommand{\conv}[1]{\mathrm{conv}\!\left(#1\right)}
\newcommand{\vol}[1]{\mathrm{vol}\!\left(#1\right)}
\newcommand{\st}{\;\colon\;}

\AtBeginDocument{%
  \mathchardef\mathcomma\mathcode`\,
  \mathcode`\,="8000
}
{\catcode`,=\active
  \gdef,{\mathcomma\discretionary{}{}{}}
}

\begin{document}

\title[Higher rank antipodality]{Higher rank antipodality}

\author[M. Nasz\'odi]{M\'arton Nasz\'odi}
\address{
Alfr\'ed R\'enyi Institute of Mathematics HUN-REN,  Budapest, Re\'altanoda u.\ 13-15, 1053 Hungary and\\
Department of Geometry, E\"{o}tv\"{o}s Lor\'{a}nd University,
P\'{a}zm\'{a}ny P\'{e}ter s\'{e}t\'{a}ny 1/C, Budapest, 1117 Hungary}
\email{marton.naszodi@renyi.hu}

\author[Zs.\ Szil\'agyi]{Zsombor Szil\'agyi}
\address{
Department of Analysis and Operations Research, Institute of Mathematics,
Budapest University of Technology and Economics M\H{u}egyetem rkp. 3--9 H-1111 
and MTA-BME Lend\"ulet ``Momentum'' Quantum Information Theory Research Group}
\email{zsombor.szilagyi@math.bme.hu}

\author[M.\ Weiner]{Mih\'aly Weiner}
\address{
Department of Analysis and Operations Research, Institute of Mathematics,
Budapest University of Technology and Economics M\H{u}egyetem rkp. 3--9 H-1111 
and MTA-BME Lend\"ulet ``Momentum'' Quantum Information Theory Research Group}
\email{mweiner@math.bme.hu}

\thanks{M.N. and M.W. were supported by the Bolyai J\'anos Fellowship of the Hungarian Academy of Sciences and the \'UNKP-22-5 New National Excellence Program of the Ministry for Innovation and Technology. 
N.M. is also supported by the NRDI grants K131529 and K147544, and the ELTE TKP 2021-NKTA-62 funding
scheme, as well as the ERC grant "Geoscape" no. 882971. M.W. is also supported by the NRDI grant K132097 and Sz.Zs. and M.W. by the Ministry of Culture and Innovation and the National Research, Development and Innovation Office within the Quantum Information National Laboratory of Hungary (Grant No. 2022-2.1.1-NL-2022-00004).}

 \subjclass[2020]{52C17}

 \keywords{Antipodality, general probability, hash function, neighborly polytope}

\begin{abstract}
Motivated by general probability theory, we say that the set $S$ in $\mathbb{R}^d$ is \emph{antipodal of rank $k$}, if for any $k+1$ elements $q_1,\ldots q_{k+1}\in S$, there is an affine map from $\conv S$ to the $k$-dimensional simplex $\Delta_k$ that maps $q_1,\ldots q_{k+1}$ bijectively onto the $k+1$ vertices of $\Delta_k$. For $k=1$, it coincides with the well-studied notion of (pairwise) antipodality introduced by Klee.
We consider the following natural generalization of Klee's problem on antipodal sets: What is the maximum size of an antipodal set of rank $k$ in $\mathbb{R}^d$? We present a geometric characterization of antipodal sets of rank $k$ and adapting the argument of Danzer and Gr\"unbaum originally developed for the $k=1$ case, we prove an upper bound which is exponential in the dimension. We show that this problem can be
connected to a classical question in computer science on finding perfect hashes, and it provides a lower bound on the maximum size, which is also exponential in the dimension. 
By connecting rank-$k$ antipodality to $k$-neighborly polytopes, we obtain another upper bound when $k>d/2$. 
\end{abstract}

\maketitle

\section{Introduction}

As introduced by Klee \cite{Klee1960}, we say that two points $q_1,q_2$ of a convex set $S$ in $\Red$ are in \emph{antipodal position} if there exist two distinct parallel hyperplanes $L_1,L_2$ supporting $S$ such that $q_1\in L_1$ and $q_2\in L_2$. Danzer and Gr\"unbaum \cite{DG1962} proved the sharp upper bound $2^d$ for the maximum size of a set $S$ in $\Red$ which is antipodal with respect to $\conv{S}$, giving also a characterization of the equality case ($S$ needs to be the vertex set of a parallelotope). In a number of variants of this notion (strict antipodality, Erd{\H o}s' notion of an obtuse-triangle-free set, see below), a stronger/weaker condition is imposed; however, in all of these, it is a
property stated for a \emph{pair} of points. Here, we consider a natural generalization to the joint position of $k\geq 2$ points of a given convex set. Our motivation for this generalization comes from general probability theory (GPT), where the notion of \emph{joint distinguishability of states} -- as was observed in \cite{LGA2022} -- for two states simplifies to that of pairwise antipodality. 

We explain this background in Section~\ref{sec:gpt} and in particular, how it leads to the following definition. For a set $S$ in $\Red$, we say that its points $q_1,\dots,q_{k+1}\in S$ are \emph{jointly antipodal with respect to $S$}, if there exists an affine map from $S$ to the $k$-dimensional simplex $\Delta_k=\left\{x\in\left(\Re_{\geq 0}\right)^{k+1}\st \sum_{i=1}^{k+1}=1\right\}$ that maps $q_1,\ldots q_{k+1}$ bijectively onto the $k+1$ vertices of $\Delta_k$. We call a set $S$ in $\Red$ \emph{antipodal of rank $k$}, if any $k+1$ points of $S$ are jointly antipodal with respect to $S$.

In Section~\ref{sec:upperbound}, we first provide a geometric characterization of joint antipodality. Let $S$ be a convex set and $q_1,\dots, q_{k+1}\in S$.
For each $j\in[k+1]=\{1,\dots, k+1\}$, dilate $S$ from center $q_j$ by factor $\lambda_j\in (0,1)$ to obtain the ``shrunk copy'' $S_j$. In Proposition~\ref{prop:antipodequivshrinking},
we use the sets $S_j$ ($j\in[k+1]$)
to characterize antipodality; in particular, we show that $q_1,\dots,q_{k+1}\in S$ are jointly antipodal with respect to $S$ if and only if, whenever the dilation factors satisfy $\lambda_1+\ldots +\lambda_{k+1}<k$, the intersection $\cap_{j=1}^{k+1}S_j$ is empty. Using this characterization, we obtain the following upper bound on the maximum size, $A(d,k)$, of a rank $k$ antipodal set in $\Red$.
\begin{thrm}[A general upper bound on the size of a rank $k$ antipodal set]\label{thm:upperbound}
For any $k\leq d$, we have $A(d,k) \leq k \left(\frac{k+1}{k}\right)^d$. 
\end{thrm}
Note that for $k>k^\prime$, rank $k$ antipodality is stronger than rank $k^\prime$ antipodality and in fact,
our upper bound decreases as $k$ increases from $1$ to $d-1$. Clearly, a rank $k$ antipodal set is of dimension at least $k$, and in case $d=k$, the set must be a subset of the $k+1$ vertices of a $k$-simplex. Note also that for $k=1$, Theorem~\ref{thm:upperbound} restates the bound of Danzer and Gr\"unbaum \cite{DG1962}; in fact, our argument is an adaptation of their idea with additional care.

Next, in Section~\ref{sec:neighborly}, we establish a connection of rank $k$ antipodality with $k$-neighborly polytopes, which will yield the following upper bound.

\begin{thrm}[Upper bound on $A(d,k)$ using $k$-neighborly polytopes]\label{thm:neighborlybound}
If $k>d/2$, then $A(d,k)=d+1$.
\end{thrm}

Finally, in Section~\ref{sec:combinatorics}, we consider constructions yielding rank $k$ antipodal sets, and in turn obtain 
lower bounds on their maximum size. We follow and expand the basic idea of the recent paper \cite{LGA2022}, which is assigning points to sequences in a suitable Cartesian product space. We note that the underlying combinatorial problem is in fact
a well-known question of computer science, namely, the construction of perfect hashes.  

For positive integers $b,k,m$ with $2<k\leq b$, a \emph{perfect $(b,k)$-hash code of length $m$} is a set $W$ of words of length $m$ on the alphabet $[b]=\{1,2,\dots,b\}$ in which for every subset $\{w_1,\dots,w_k\}$ of $k$ elements of $W$, there is a $j\in[m]$ such that the $j^{\rm th}$ letters of the words $w_1,\dots,w_k$ are all different. 
Let $N(b,k,m)$ denote the size of the largest perfect $(b,k)$-hash code of length $m$. A long-studied problem of computer science \cite{FC84, KoMa88, DCD22, XiYu21} is the determination of $N(b,k,m)$, or at least
-- since we know that for fixed $b,k$ it grows exponentially with the length $m$ -- its 
asymptotic rate, i.e.\ the quantity
\begin{equation}
\label{eq:R}
R(b,k)= \limsup_{m\to \infty}\frac{1}{m}\log (N(b,k,m)).
\end{equation}
Since in general, only lower and upper bounds are known for these quantities, 
we leave them in their ``unevaluated'' form and state that the mentioned construction yields the following lower bound for our geometric quantity $A(d,k)$.
\begin{thrm}[Hashes yield rank $k$ antipodal sets]\label{thm:combimpliesgeo}
Assume that there is a rank $k$ antipodal set in dimension $d_0$ of size $b$. Then for every $m=1,2,\ldots$ one can construct a rank $k$ antipodal set in dimension $d=m\cdot d_0$ of size $N(b,k+1,m)$.
\end{thrm}

In essence, the same underlying construction is presented in \cite{LGA2022} using the vertices of a $d_0$-simplex as the ``starting configuration''. The aim in that paper was only to prove that for any fixed $k$, here exist rank $k$ antipodal sets whose size grows
exponentially with the dimension; see \cite[Theorem~17]{LGA2022}. We point out the fact that the combinatorial problem is well-studied in computer science, and improve the results in \cite{LGA2022} in three ways. First, even when using a simplex for a starting configuration  (which is what they do), one could ``optimize'' for its dimension $d_0$ (which they did not attempt). 

For example, consider creating antipodal sets of rank $k=2$ using the above construction starting from the simplex of dimension $d_0\geq 2$. Using the lower bounds in \cite{KoMa88}, we then find that
$$
\liminf_{d\to \infty}\frac{1}{d}\log A(d,2) 
\geq \frac{1}{2d_0} \log\frac{(d_0+1)^2}{3d_0+1}.
$$
Standard calculus arguments show that the right hand side of the above inequality is the largest (among integers greater than 1) for $d_0=4$, showing that asymptotically, $A(d,2)$ grows at least as $1.085^d$. 

Second, the simplex may be far from being an optimal starting configuration. For example, when $k=2$, then neither Theorem \ref{thm:upperbound}
nor Theorem \ref{thm:neighborlybound} does not exclude the existence of a rank $2$ antipodal set in dimension $d_0=4$ of size $6$ (i.e.\! larger than the number of vertices of the four dimensional simplex), with which we would of course ``beat'' the previously mentioned growing rate of $1.085^d$.
Third, there are better random constructions \cite{KoMa88} for the underlying combinatorial problem and in fact, constructions beating the known random ones \cite{XiYu21} in their exponential rates of growth.

In this way, various lower bounds of the form $c \cdot \alpha^d$ can be given for the possible size of a rank $k$ antipodal set in dimension $d$ with both the base $\alpha>1$ and the multiplicative constant $c>0$ depending on $k$. However, as we shall note, for $k>2$, this construction \emph{cannot} give a lower bound with $\alpha=(k+1)/k$, -- which is what appears in our upper bound Theorem~\ref{thm:upperbound} -- regardless of any later improvements on the lower bound of $R(b,k)$. Thus, a gap remains between our upper bound, Theorem~\ref{thm:upperbound} on the geometric quantity $A(d,k)$, and the constructions that one can obtain using perfect hashes, a combinatorial notion. The gap exist perhaps because there are better -- e.g.\ fundamentally geometric -- constructions of rank $k$ antipodal sets, or because our upper bound is not optimal. In the remaining part of this introduction, we pose some problems.

\subsection{Questions}

The first, obvious question is to close the gap between the lower and the upper bounds on the maximum size $A(d,k)$ of an antipodal set of rank $k$ in $\Red$ for $k>1$. We found no reason to believe that either the lower or the upper bound that we present should be even asymptotically close to optimal. It would be very interesting to find out whether the answer to the geometric problem (the maximum size of a rank $k$ antipodal set) and the answer to the combinatorial problem (size of a perfect hash) are of the same or of different orders of magnitude.

We defined antipodality of rank $k$ following Klee, however, one could generalize a strongly related question of Erd{\H o}s \cite{E1957}, where the maximum size of a set in $\Red$ is to be determined with the property that it does not contain the vertices of an obtuse triangle. In other words, for any pair of points $q_1,q_2$ of the set, the two hyperplanes perpendicular to the line segment $[q_1,q_2]$, one through each end point, support the set. Generalizing this definition, we may call a set $S$ of points in $\Red$ \emph{Erd{\H o}s-antipodal of rank $k$}, if for any $k+1$ elements $q_1,\dots,q_{k+1}\in S$ the orthogonal projection of $\Red$ onto the $k$-dimensional affine hull $\aff{q_1,\dots,q_{k+1}}$ maps $S$ to the simplex $\conv{q_1,\dots,q_{k+1}}$. One may investigate the maximum size of such a set in $\Red$. For $k=1$, we know that Klee's notion of antipodality, and Erd{\H o}s' stronger notion yield the same bound, $2^d$, as shown by Danzer and Gr\"unbaum \cite{DG1962} (and the example of the cube).

Finally, \emph{strictly antipodal sets}, as introduced by Gr\"unbaum \cite{Gru1963}, are those antipodal sets in $\Red$ where for any pair of points $q_1$ and $q_2$ of the set, the pair of distinct parallel hyperplanes may be chosen in such a way that each intersects the set in exactly one point, namely in $q_1$ and $q_2$ respectively; cf. \cite{Za2019, GH2019} for recent results and \cite{GH2019s} for a survey.
By analogy, we may call a set $S$ \emph{strictly antipodal of rank $k$}, if for any $k+1$ elements $q_1,\dots,q_{k+1}\in S$ there is a projection of $\Red$ onto a $k$-dimensional plane, and there is a $k$-dimensional simplex $\Delta_k$ in this plane such that $S$ is mapped to $\Delta_k$, the set $\{q_1,\dots,q_{k+1}\}$ is mapped bijectively onto the vertices of $\Delta_k$, and the image of no other point of $S$ is a vertex of $\Delta_k$. Lower and upper bounds for the maximum size of such set in $\Red$ would be of interest.

\section{Motivation: General Probability Theory}\label{sec:gpt}

General probability theory (GPT) was developed
to provide a common framework to classical
and quantum probabilistic models and to
explore what options (compatible
with some natural physical requirements)
could exists beyond quantum probability, see e.g.\ \cite{gptbackground1,gptbackground2}.
In this section, we  give a condensed, simplified
version of what a GPT model is, concentrating only on 
the concepts of \emph{states} and \emph{measurements} (and omitting many other notions, eg. effects, composite systems, etc.).

\subsection{A brief introduction to GPT}

From a probabilistic point of view, to describe a physical system, first, one needs to choose a set 
$S$ which will be viewed as the set of possible \emph{states}. Next, we have to introduce the concept of  
a measurement with finitely many, say $k+1$, possible outcomes identified by / corresponding to the numbers $1,2,\dots, k+1$. What we want is to be able to
talk about the outcome-statistics; i.e.\ the probabilities of obtaining outcomes $1,2,\dots, k+1$, or collected into a $k+1$-tuple, an element of the standard simplex $\Delta_k=\{(p_1,\ldots p_{k+1})\in \mathbb R^{k+1} \st \, \sum_{j}p_j=1, \, p_j\geq 0 \text{ for all } j\}$. Since this statistics may depend on the actual state, in our model a measurement with $k+1$ numbered possible outcomes is a function from the set of states $S$ to $\Delta_k$. Thus the model should minimally consists of a set $S$ and a collection of functions from $S$ to some (possibly distinct dimensional) simplices corresponding to the measurements that can be performed on the system.

An important point is that physically it is meaningful to consider the convex combination of states. Operationally, to obtain outcome statistics of a measurement $M$ in the ``mixed state'' $\lambda s_1 +(1-\lambda)s_2$ where $s_1,s_2\in S$ and the coefficients $\lambda$ and $1-\lambda$ are in $[0,1]$, one prepares the system in either state $s_1$ or $s_2$ with corresponding probabilities $\lambda$ and $1-\lambda$, then performs $M$.
Thus, the state space should be endowed with a structure making it a \emph{convex set} and moreover, by this operational meaning of convex combinations of states, every measurement should be an affine map from $S$ to a simplex.

Usually, when defining what a GPT model is, one assumes the \emph{no restriction principle} \cite{norestr}: every ``theoretically possible measurement'', that is, every affine map from $S$ to a simplex, is actually a realizable measurement. So if we only wish to talk about states and measurements, then under this assumption, the whole probability model is completely determined by the choice of the convex set $S$ playing the role of the state space of the physical system. This principle has physical motivations, and it holds in both the classical (where in the ``finite level'' we set $S=\Delta_n$ for some $n$) and in the quantum case (where in the finite level, $S$ is the set of density operators on $\mathbb C^n$ for some $n$). So in what follows, we shall always assume it when considering a GPT model.

\subsection{Joint distinguishability and joint antipodality}

Assume that a physical system's state space is modeled by the convex set $S$ and let 
$s_1,s_2,\ldots s_{k+1}$ be some specific elements of $S$. Aiming to store 
information (to use the system as a memory device),
the system is put into state $s_j$ according
to some selected value $j=x\in[k+1]$. To retrieve the value of $x$, we need
to establish if the system is in state $s_1$, $s_2\ldots$ or in $s_{k+1}$.

In order to distinguish, we perform an appropriate measurement with $k+1$ possible outcomes and take its outcome as our hypothesis for the value of $x$. Suppose this measurement is modeled by the affine map $M:S\to \Delta_k$. Then, given that $x=j$, the probability that the measurement will end with its $j^{\rm th}$ outcome (i.e.\ it correctly indicates the value of $x$) is $M_j(s_j)$, and it is precisely $1$ if and only if $M(s_j)$ is the $j^{\rm th}$ vertex
of $\Delta_k$. If there exists an affine map $M:S\to \Delta_k$ (i.e.\ a measurement) such that this holds for every $j\in[k+1]$, then $s_1,s_2,\ldots s_{k+1}$ are {\it jointly (perfectly) distinguishable}. 
Note that this is equivalent to the existence of an affine map $M$ from $S$ to {\it some} simplex such that $M(s_1), M(s_2),\ldots M(s_{k+1})$ are 
$k+1$ distinct vertices of the simplex in question, which is what we adopted as a definition of joint antipodality.

As pointed out in \cite{LGA2022}, small dimensional rank $k$ antipodal sets with large cardinality are interesting from an information theoretic point of view. Assume that we want to build a memory device which may store an integer $x$, say, between 1 and 1024. Assume further that we do not need to be able to retrieve the value $x$. Instead, for any three-element subset $H$ of $[1024]$ that contains $x$, we want to be able to tell which element of $H$ $x$ is. If the set of states of our memory device is the convex set $S$, then according to the argument above, in the classical (when $S$ is a simplex) and in the quantum (when $S$ is the set of density operators on $\mathbb C^n$ for some $n$) settings, we need 10 bits, or, in geometric terms, we need a least $(2^{10}-1)$-dimensional space to embed $S$ in. However, if $S$ is a smaller dimensional rank $3$ antipodal polytope with at least 1024 vertices, then we have a smaller memory performing the job. In other words, $S$ makes a certain kind of \emph{data compression} possible.

\section{Equivalent descriptions of joint antipodality and volume-bounds}\label{sec:upperbound}

Given a point $q\in \Red$ and $\lambda\in \R$, we denote the \emph{dilation} from center $q$ by factor 
$\lambda$ by $D_{q,\lambda}$, that is,
\[
D_{q,\lambda}(x) = (1-\lambda) q + \lambda x,\;\;\; (x\in \Red).
\]
We denote by $\relint{S}$ the \emph{relative interior} of a set $S\subset \Red$; i.e.\ the interior of $S$ with respect to the affine hull $\aff{S}$ of $S$.

\begin{prop}[Characterization of joint antipodality in terms of dilates]\label{prop:antipodequivshrinking}
Let $S$ be a convex set in $\Red$ and $q_1,\ldots q_{k+1}\in S$. Then the following are equivalent.
\begin{enumerate}
    \item\label{item:antipod} 
$q_1,\ldots q_{k+1}$ are jointly antipodal with respect to $S$;
    \item\label{item:separatedall}
for any $\lambda_1,\ldots \lambda_{k+1}\in(0,1)$ with $\lambda_1+\ldots +\lambda_{k+1}=k$, we have 
\[
\mathop{\cap}_{j=1}^{k+1} D_{q_j,\lambda_j}(\relint{S}) = \emptyset;
\] 
    \item\label{item:separated}
for some $\lambda_1,\ldots \lambda_{k+1}\in(0,1)$ with $\lambda_1+\ldots +\lambda_{k+1}=k$, we have
\[
\mathop{\cap}_{j=1}^{k+1} D_{q_j,\lambda_j}(\relint{S}) = \emptyset.
\] 
\end{enumerate}
\end{prop}
 
Note that we could avoid using relative interiors; e.g.\ condition \eqref{item:separatedall} is equivalent to saying that the intersection $\cap_{j=1}^{k+1} D_{q_j,\lambda_j}(S)$ is empty for any $\lambda_1,\ldots \lambda_{k+1}\in (0,1)$ such that $\lambda_1 + \ldots + \lambda_{k+1}<k$. We postpone the proof of the Proposition, and first show how this characterization implies Theorem~\ref{thm:upperbound}.

Proposition~\ref{prop:antipodequivshrinking} immediately yields the following.
\begin{cor}\label{cor:antipodequivshrinking}
    A set $S$ in $\Red$  is antipodal of rank $k$ if and only if, the sets $D_{x,k/(k+1)}(K)$ with $x \in S$, cover each point of $\Red$ at most $k$-fold.
\end{cor}

\begin{proof}[Proof of Theorem~\ref{thm:upperbound}]
Clearly, we may assume that $\aff{q_1,\dots,q_n}$ is $\Red$.
Set $S=\conv{q_1,\dots,q_n}$, and consider the sets $S_j=D_{q_j,k/(k+1)}(\inter{S})$.
Clearly, each $S_j$ is contained in $S$, and $\vol{S_j}=\left(\frac{k}{k+1}\right)^d\vol{S}$.
By Corollary~\ref{cor:antipodequivshrinking}, every point of $S$ is contained in at most $k$ of the $S_j$. Thus,
\[
 \sum_{j=1}^{n}\vol{S_j}\leq k\cdot \vol{S}.
\]
The statement of Theorem~\ref{thm:upperbound} follows.
\end{proof}
The main idea of the proof above, that is, to consider smaller dilates of $S$ inside $S$ and compute volume is due to Danzer and Gr\"unbaum \cite{DG1962}, our adaptation relies on the observation that it can be applied in the present context using an arrangement of dilates that do not cover any point more than  $k$-fold.

The proof of Proposition~\ref{prop:antipodequivshrinking} requires a lemma, which is a slightly modified version of \cite[Theorem~7.1]{barany15}. For completeness, we outline its proof.
\begin{lemma}[Separation of multiple convex sets]\label{lem:barany}
Let $K_1,\dots,K_r$ be convex bodies in $\R^k$, that is, compact convex sets with non-empty interior. Assume that $\cap_{i=1}^r\inter{K_i}=\emptyset$ and $p\in\cap_{i=1}^r K_i$. Then there are closed half-spaces $D_1,\dots,D_r$ containing $p$ on their boundary with \linebreak $\cap_{i=1}^r\inter{D_i} = \emptyset$ and $K_i\subset D_i$ for all $i\in[r]$. 
\end{lemma}
\begin{proof}[Proof of Lemma~\ref{lem:barany}]
First, consider the compact convex set $K_1$ and the closed convex set $\cap_{i=2}^r K_i$. By a standard separation theorem, there is a closed half-space $D_1$ containing $p$ on its boundary with $\inter{K_1}\subset \inter{D_1}$ and $\inter{D_1}\bigcap \cap_{i=2}^r K_i=\emptyset$. We replace $K_1$ with $D_1$ and apply the same argument for the compact convex set $K_2$ and the closed convex set $D_1\bigcap \cap_{i=3}^r K_i$ to obtain $D_2$. 
Next, we apply the same argument for the compact convex set $K_3$ and the closed convex set $(D_1\cap D_2)\bigcap \cap_{i=4}^r K_i$ to obtain $D_3$. Continuing in this manner, we finally apply the same argument for the compact convex set $K_r$ and the closed convex set $\cap_{i=1}^{r-1} D_i$ to obtain $D_r$. As is easily verified, we arrive at a desired set of half-spaces.
\end{proof}

The following lemma is a simple exercise nevertheless, we sketch the proof.
\begin{lemma}[Non-overlapping dilates of a simplex]\label{lem:simplexparts}
 Let $v_1,\dots,v_{k+1}$ denote the vertices of the simplex $\Delta_k$ in $\R^k$. Then for any $\lambda_1,\ldots \lambda_{k+1}\in(0,1)$ with $\lambda_1+\ldots +\lambda_{k+1}=k$, we have 
\[
\mathop{\cap}_{j=1}^{k+1} D_{v_j,\lambda_j}(\inter{ \Delta_k}) = \emptyset;
\]
\end{lemma}
\begin{proof}
 Any point $x$ in $\R^k$ can be written as $x=\sum_{j=1}^{k+1}\mu_j v_j$ with uniquely determined real coefficients $\mu_j$ such that $\sum_{j=1}^{k+1}\mu_j = 1$. These $\mu_j$ are called the barycentric coordinates of $x$ with respect to the $v_j$. Observing
 \[
  D_{v_j,\lambda_j}(\inter{ \Delta_k})=\{x\in\inter{ \Delta_k}\st \mu_j>1-\lambda_j\}
 \]
and $\sum_{j=1}^{k+1} (1-\lambda_j) = 1$ yields the proof.
\end{proof}

\begin{proof}[Proof of Proposition~\ref{prop:antipodequivshrinking}]
We may assume that $S$ has non-empty interior, otherwise, we replace $\Red$ by $\aff{S}$.

In order to prove that \eqref{item:antipod} implies \eqref{item:separatedall}, assume that there is an affine map $\phi$ from $S$ to $\Delta_k$ mapping $q_1,\ldots q_{k+1}$ onto the vertices of $\Delta_k$. Let $\lambda_1,\ldots \lambda_{k+1}\in(0,1)$ with $\lambda_1+\ldots +\lambda_{k+1}=k$, and observe that
\[
\mathop{\cap}_{j=1}^{k+1} \phi\left(D_{q_j,\lambda_j}(\relint{S})\right) = 
\mathop{\cap}_{j=1}^{k+1} D_{\phi(q_j),\lambda_j}(\relint{\phi(S)}) = 
\]\[
\mathop{\cap}_{j=1}^{k+1} D_{\phi(q_j),\lambda_j}(\relint{\Delta_k}) = 
\emptyset,
\] 
where, in the last equation, we used Lemma~\ref{lem:simplexparts}. Thus, \eqref{item:separatedall} holds.

Clearly, \eqref{item:separatedall} implies \eqref{item:separated}. 
In order to see that \eqref{item:separated} implies \eqref{item:antipod}, we first observe that the point $p\coloneqq \sum_{i=1}^{k+1}(1-\lambda_i)q_i$ is in the intersection $\cap_{i=1}^{k+1} D_{q_i,\lambda_i}(S)$. 
Indeed, to see that $p\in D_{q_1,\lambda_1}(S)$, we write
\[p=(1-\lambda_1)q_1+\lambda_1\sum_{i=2}^{k+1}\frac{1-\lambda_i}{\sum_{j=2}^{k+1}(1-\lambda_j)}q_i,
\]
using $\lambda_1=\sum_{i=2}^{k+1}(1-\lambda_i)$. Similarly, $p\in D_{q_i,\lambda_i}(S)$ for all $i\in[k+1]$.

Applying Lemma~\ref{lem:barany} yields closed half-spaces $D_i$ $(i\in[k+1])$ whose bounding hyperplanes, denoted by $H_i$, are incident with $p$. Note that for any distinct $i,j\in[k+1]$ we have that $q_i+\lambda_i(q_j-q_i)$ is in $D_{q_i,\lambda_i}(S)$. Moreover, for any $i\in[k+1]$, we have
$p=\sum_{j\in[k+1]\setminus\{i\}]} \frac{1-\lambda_j}{\lambda_i}(q_i+\lambda_i(q_j-q_i))$ (the reader is invited to check this), and hence $p$ is in the relative interior of the convex hull of the set $A_i\coloneqq\{q_i+\lambda_i(q_j-q_i)\st j\in[k+1]\setminus\{i\}\}$. Since $H_i$ is a support hyperplane of $D_{\lambda_i,q_i}(S)$ at $p$, and $A_i\subseteq D_{\lambda_i,q_i}(S)$, it follows that $A_i\subset H_i$.

It follows that the hyperplane $D_{q_i,1/\lambda_i}(H_i)$ supports the set $S$ at the points $q_j$ with $j\in [k+1]\setminus\{i\}$. Hence, the intersection $\cap_{i=1}^{k+1} D_{q_i,1/\lambda_i}(D_i)$ of $k+1$ closed half-spaces contains $S$, and $\aff{q_1,\dots,q_{k+1}}\bigcap\cap_{i=1}^{k+1} D_{q_i,1/\lambda_i}(D_i)$ is a $k$-dimensional simplex which is equal to $\conv{q_1,\dots,q_{k+1}}$. Moreover, since $\cap_{i=1}^{k+1} \inter{D_i}=\emptyset$ and $\cap_{i=1}^{k+1} D_i$ contains $p$, and hence is not empty, the normal vectors of the hyperplanes $H_i$ are linearly dependent. Thus, the dimension of the affine subspace $\cap_{i=1}^{k+1} D_{q_i,1/\lambda_i}(H_i)$ is $d-k$. Finally the projection of $\Red$ onto $\aff{q_1,\dots,q_{k+1}}$ along the linear subspace parallel to $\cap_{i=1}^{k+1} D_{q_i,1/\lambda_i}(H_i)$ is the desired affine map satisfying \eqref{item:antipod}.
\end{proof}

\section{Neighborly polytopes}\label{sec:neighborly}

Fixing a $k\in[d]$, we say that a convex polytope $P$ in $\Red$ is \emph{$k$-neighborly}, if for any $k$-element subset $A$ of the set of vertices of $P$, the convex set $\conv{A}$ is a face of $P$. Members of this class of polytopes possess a
very rich combinatorial structure, Gr\"unbum's \cite[Chapter~7]{Gru03} provides an excellent introduction to this theory.

\begin{thrm}[$k$-antipodal sets are $k$-neighborly polytopes]\label{thm:neighborly}
Let $S$ be an antipodal set of rank $k$ in $\Red$ for some $k\in[d]$. Then $\conv{S}$ is a $k$-neighborly convex polytope.
\end{thrm}

The cube ($[-1,1]^d$) is rank $1$ antipodal, but clearly not 2-neighborly if $d\geq 2$, so the theorem in this sense is sharp.

\begin{proof}
 Let $A$ be a $k$-element subset of $S$. Denote by $F$ the smallest face (with respect to inclusion) of $P=\conv{S}$ that contains $A$. If $A$ is the vertex set of $F$, we are done, so suppose this is not the case.

 Note that since $S$ is antipodal of rank $k$, it is also antipodal of rank $k-1$.
 Clearly, $\dim{F}>\dim{A}$, as otherwise $\aff{A}$ would contain a vertex of $P$ not in $A$, clearly contradicting the assumption that there is an affine map of $\Red$ to the simplex $\Delta_{k-1}$ mapping $A$ to bijectively to the vertices of $\Delta_{k-1}$.

 Set $t=\sum_{a\in A} a/k$, that is, the center of mass of the
$(k-1)$-dimensional simplex $\conv{A}$. We claim that $t$ is in the relative
interior of $F$, that is, it is in the interior of $F$ with respect to
$\aff{F}$. Indeed, otherwise $F$ had a supporting hyperplane $L$ within $\aff{F}$ containing $t$, which would in turn, clearly, contain $A$, and hence, $L\cap F$ would be a face of $P$ containing $A$ which is strictly contained in $F$, contradicting the minimality of $F$.

Next, let $q$ be any vertex of $F$ outside of $\aff{A}$. The point $t$ is on the relative boundary of the $k$-dimensional simplex $\conv{A\cup\{q\}}$, but is in the relative interior of $F$. Clearly, there is no affine map from $F$ to $\Delta_k$ that maps $A\cup\{q\}$ bijectively to the vertices of $\Delta_k$, a contradiction completing the proof of Theorem~\ref{thm:neighborly}.
\end{proof}

Theorem~\ref{thm:neighborlybound} now follows immediately from the fact that any $k$-neighborly polytope in $\Red$ with $k>d/2$ contains at most $d+1$ vertices, that is, it must be a simplex, cf. \cite[Section~7.1]{Gru03}.

\section{Hash functions yield rank \texorpdfstring{$k$}{k} antipodal sets}\label{sec:combinatorics}

\begin{proof}[Proof of Theorem~\ref{thm:combimpliesgeo}]
 Let $S\subset \mathbb R^{d_0}$ be a rank  $k$ antipodal set of cardinality $|S|=b$,
and $W_m$ be a perfect $(b,k+1)$-hash code of words of length $m$, with cardinality $|W_m|=N(b,k+1,m)$. 
Consider the $m$-th Cartesian power of $S_0=\conv{S}$, $S=\overbrace{S_0\times\dots\times S_0}^{m \text{ times}}$ as a convex subset of $\R^{md_0}$.
Clearly, $S$ is a polytope whose vertices can be encoded as words of length $m$ on the alphabet $[b]$. That is, $W_m$ may be considered as a subset $X_{W_m}$ of the vertices of $S$. It is easy to see that by the definition of a perfect hash code, $X_{W_m}$ is antipodal of rank $k$ with respect to $S$. In fact, the projections needed to show this belong to the set $\{P_i\st i\in[m]\}$, where $P_i$ is the projection of $\Red=\left(\R^{d_0}\right)^m$ to its $i$-th component $\R^{d_0}$.
\end{proof}

\section{Closing Remarks}

\subsection{The gap between geometry and combinatorics}

One might wonder if the construction in the proof of Theorem~\ref{thm:combimpliesgeo} can lead to a sequence of examples approaching, in an asymptotic sense, the upper bound on $A(d,k)$. More specifically, whether for any fixed $k$, there is a \emph{gap} between the bound
\[
\limsup_{d\longrightarrow\infty} \frac{1}{d}\log A(d,k) \leq \log\frac{k+1}{k}
\]
given by Theorem~\ref{thm:upperbound}, and the quantity
\begin{equation}\label{eq:limsuplog}
\limsup_{m\to \infty}\frac{1}{d}\log(|X_{W_m}|),    
\end{equation}
where $X_{W_m}$ is the rank $k$ antipodal set obtained in the proof of Theorem~\ref{thm:combimpliesgeo}.

We now show that for $k>2$, in fact, \emph{there is a gap independent of the starting configuration.} That is, using prefect hashes to obtain high dimensional rank $k$ antipodal sets from a fixed $d_0$-dimensional rank $k$ antipodal example, one cannot get close even asymptotically to the geometric bound on $A(d,k)$ given in Theorem~\ref{thm:upperbound}.

Indeed, by our result, we know that for the ``starting configuration'' $S$, we have the inequality
\begin{equation}\label{eq:boundonb}
b \leq k \left(\frac{k+1}{k}\right)^{d_0}.    
\end{equation}
Moreover, a trivial counting argument c.f. \cite{KoMa88} gives an upper bound on the cardinality of a perfect hash code by which
\begin{equation}\label{eq:boundonN}
N(b,k+1,m)\leq k \left(\frac{b}{k}\right)^{m}.
\end{equation}
If we assume equality both in \eqref{eq:boundonb} and \eqref{eq:boundonN}, then after substituting $d=d_0 m$, we obtain that our rank $k$ antipodal set $X_{W_m}$ constructed in the proof of Theorem~\ref{thm:combimpliesgeo} has cardinality $|X_{W_m}|=|W_m|=k((k+1)/k)^d$, which is again precisely the upper bound of Theorem~\ref{thm:upperbound}.
On the other hand, $k\left(\frac{k+1}{k}\right)^{d_0}$ is never an integer (for $d_0,k>1$), so it cannot be equal to $b$, from which it is easy to conclude that for any fixed starting set, the quantity \eqref{eq:limsuplog} is strictly smaller than $\log \frac{k+1}{k}$. Nevertheless, it does not rule out the possibility of \eqref{eq:limsuplog} getting arbitrarily close to $\log \frac{k+1}{k}$ as the dimension $d_0$ of the starting configuration increases.

Luckily, in general we have  better bounds than \eqref{eq:boundonN} for the asymptotic of the size of a perfect hash code; see for example \cite{KoMa88, FC84, DCD22}.
Using the quantity $R$ introduced in \eqref{eq:R}, we have 
\[
\limsup_{m\to \infty}\frac{1}{d}\log(|X_{W_m}|) = \frac{1}{d_0}R(b,k+1)
\leq \frac{1}{d_0}R\left(k \left(\frac{k+1}{k}\right)^{d_0}\!,\, k+1\right).
\]
For $k>2$, and $d_0$ sufficiently large, using \cite[Lemma~1]{DCD22} for upper estimating $R$, it can be shown that \eqref{eq:limsuplog} may be upper bounded by a value depending only on $k$ (and not on $d_0$), which is strictly smaller than $\log \frac{k+1}{k}$. The existence of a gap then follows, since for each $k$, there is only a finite number of small $d_0$.



\subsection{Weakening joint perfect distinguishablity in GPT}

If the states $s_1,s_2,\ldots, s_{k+1}\in S$ in a GPT model are not jointly perfectly distinguishable, one may minimize the sum of the error probabilities
$$
P_{e,M}(s_1,s_2,\ldots, s_{k+1})=\sum_{j=1}^{k+1}(1-M_j(s_j))
$$
and investigate how its infimum $P_e^*(s_1,\ldots, s_{k+1})$ over all possible
affine maps $M:S\to \Delta_k$ behaves. For example, it is interesting to find upper bounds on $P_e^*(s_1,\ldots, s_{k+1})$ in terms of the pairwise quantities $P_e^*(s_j,s_{\ell})$, $j\neq\ell$. If $S=\Delta_n$ (ie., in classical probability), then evidently $P_e^*(s_1,\ldots, s_{k+1})\leq \sum_{j<\ell}P_e^*(s_j,s_\ell)$. On the other hand, it is false in general when $S$ is the set of
density operators on $\mathbb C_n$ for some $n$ (i.e.\ in the quantum probability setting). Nevertheless, by \cite[equation (10)]{audenaert_milan} in this latter case, one still has the (somewhat weaker) bound
$P_e^*(s_1,\ldots, s_{k+1})\leq 2\sqrt{2}\sum_{j<\ell}\sqrt{P_e^*(s_j,s_\ell)}$.

Interestingly, both in the quantum and in the classical case (as one might deduce from the upper bounds on $P_e^*$ in the paragraph above), pairwise perfect distinguishability (or, in the language of convex geometry: pairwise antipodality) implies joint distinguishability. In general, when $S$ is an arbitrary convex set, this is not so; in fact for any collection of subsets $\mathcal A$ of $[k+1]$ forming an \emph{independence system}, there exists a convex polytope $S\subset \Red$ and $s_1,\ldots s_{k+1}\in S$ such that for an $H\subset [k+1]$, the collection $(s_j)_{j\in H}$ is jointly antipodal if and only if $H\in \mathcal A$, see the details in \cite{weiner2023}.

\subsection*{Acknowledgment} We thank Viktor Harangi for fruitful discussions on the subject.

\bibliographystyle{amsalpha} 
\bibliography{biblio}
\end{document}